\documentclass{imanumARXIV}
\usepackage{graphicx}
\usepackage{subfig}
\usepackage{verbatim}

\newcommand{\norm}[1]{\| #1 \|}
\newcommand{\barnorm}[1]{\|\mkern-1.5mu| #1 |\mkern-1mu\|}
\newcommand{\hbarnorm}[1]{\|#1\|_H}

\newcommand{\calE}{\mathcal{E}}
\newcommand{\calT}{\mathcal{T}}

\newcommand{\dA}{\ \textrm{dA}}
\newcommand{\ds}{\ \textrm{ds}}

\newcommand{\dAh}{\ \textrm{dA}_{\textrm{h}}}

\newcommand{\dsh}{\ \textrm{ds}_{\textrm{h}}}

\newcommand{\half}{\frac{1}{2}}

\newcommand{\WT}[1]{\widetilde{#1}}
\newcommand{\WH}[1]{#1}
\newcommand{\ADG}{\mathcal{D}_{h}}
\newcommand{\BDG}{\mathcal{B}_{h}}
\newcommand{\Cw}{\mathcal{C}_{w_h}}
\newcommand{\CDG}{\mathcal{A}_{h}}
\newcommand{\AC}{\mathcal{D}}
\newcommand{\BC}{\mathcal{B}}

\newcommand{\CDGl}{\mathcal{A}}

\begin{document}

\title{Discontinuous Galerkin methods for hyperbolic and advection-dominated problems on surfaces}
\shorttitle{DG methods for advection-dominated problems on surfaces}

\author{%
  Andreas Dedner \thanks{Corresponding author, Email:
a.s.dedner@warwick.ac.uk}
{\sc and
Pravin Madhavan} \\[2pt]
Mathematics Institute and Centre for Scientific Computing, University of Warwick,\\
Coventry CV4 7AL, UK
}
\shortauthorlist{Dedner \emph{et al.}}

\maketitle

\begin{abstract}
{
We extend the discontinuous Galerkin (DG) framework to the analysis of first-order hyperbolic and advection-dominated problems posed on implicitely defined surfaces. The focus will be on the hyperbolic part, which is discretised using a ``discrete surface'' generalisation of the jump-stabilised upwind flux considered in \cite{brezzi2004discontinuous}.  A key issue arising in the analysis (which does not appear in the planar setting) is the treatment of the discrete velocity field, choices of which play an important role in the stability of the scheme. We then prove optimal error estimates in an appropriate norm given a number of assumptions on the discrete velocity field, which are then investigated and discussed in more detail. The theoretical results are verified numerically for a number of test problems exhibiting advection-dominated behaviour. 
}
{discontinuous galerkin; upwind; surface partial differential equations; hyperbolic partial differential equations; advection-dominated problems}
\end{abstract}

\section{Introduction}
Partial differential equations (PDEs) on manifolds have become an active area of research in recent years due to the fact that, in many applications, mathematical models have to be formulated not on a flat Euclidean domain but on a curved surface. For example, they arise naturally in fluid dynamics (e.g.,~surface active agents on the interface between two fluids, \cite{JamLow04}) and material science (e.g.,~diffusion of species along grain boundaries, \cite{DecEllSty01}) but have also emerged in other areas as image processing and cell biology (e.g.,~cell motility involving processes on the cell membrane, \cite{neilson2011modelling} or phase separation on biomembranes, \cite{EllSti10}). 

Finite element methods (FEMs) for elliptic problems and their error analysis have been successfully applied to problems on surfaces via the intrinsic approach in \cite{dziuk1988finite}. This approach has subsequently been extended to parabolic problems \cite{dziuk2007surface} as well as evolving surfaces \cite{dziuk2007finite}. The literature on the application of FEM to various surface PDEs is now quite extensive, a review of which can be found in \cite{dziuk2013finite}. High order error estimates, which require high order surface approximations, have been derived in \cite{demlow2009higher} for the Laplace-Beltrami operator. However, there are a number of situations where conforming FEMs may not be the appropriate numerical method, for instance, problems which lead to steep gradients or even discontinuities in the solution. Such issues can arise for problems posed on surfaces, as in \cite{sokolov2012numerical} where the authors analyse a model for bacteria/cell aggregation. Without an appropriate stabilisation mechanism artificially added to the surface FEMs scheme, the solution can exhibit a spurious oscillatory behaviour which, in the context of the above problem, leads to negative densities of on-surface living cells.

Given the well-known in-built stabilisation mechanisms discontinuous Galerkin methods possess for dealing with hyperbolic/advection dominated problems and solution blow-ups, it is natural to extend the DG framework for PDEs posed on surfaces. DG methods have first been extended to surfaces in \cite{dedner2012analysis}, where an interior penalty (IP) method for a linear second-order elliptic problem was introduced and optimal a priori error estimates in the $L^{2}$ and energy norms for piecewise linear ansatz functions and surface approximations were derived. These results were then generalised to both a more general class of surface DG methods and higher order surface approximations in \cite{antonietti2013analysis}.

In this paper, we will extend the analysis of the surface DG method to the
model problem
\begin{align}
\label{eq:modelproblem}
- \epsilon \Delta_{\Gamma}u + \nabla_{\Gamma} \cdot ( w u ) + cu = f \quad \mbox{on}\ \Gamma,
\end{align} 
where $\epsilon$ is assumed to be small or even equal to zero.
For the elliptic part (when present) we can make use of the surface DG framework described in 
\cite{antonietti2013analysis}. Our main focus will thus be on the
hyperbolic part. The corresponding model problem takes the form 
\begin{align*}
\nabla_{\Gamma} \cdot ( w u ) + u = f \quad \mbox{on}\ \Gamma,
\end{align*} 
where $\Gamma$ is a compact smooth oriented surface in $\mathbb{R}^{3}$,
$c>0$ and $w$ is a velocity-field which is purely tangential to the
surface $\Gamma$. This advection problem is discretised using a ``discrete surface'' generalisation of the 
jump-stabilised upwind flux considered 
in \cite{brezzi2004discontinuous}. A number of challenging issues which do not appear in the planar setting arise when 
attempting to prove stability of the numerical scheme, related to the treatment of the velocity field on the discrete surface. 
We derive optimal a priori error estimates for this scheme given a number of assumptions on the discrete velocity field. 
We then justify these assumptions by choosing the discrete velocity field to be a Raviart-Thomas-type interpolant 
of the velocity field. Numerical results are then presented 
for test problems exhibiting advection-dominated behaviour, suggesting that our surface DG method is stable and 
free of spurious oscillations.

\section{Notation and setting}
\subsection{Continuous surface $\Gamma$}
Let $\Gamma$ be a compact smooth and oriented surface in $\mathbb{R}^{3}$ given by the zero level-set of a signed distance function $|d(x)| = dist(x,\Gamma)$ defined in an open subset $U$ of $\mathbb{R}^{3}$. For simplicity we assume that $\partial \Gamma = \emptyset$ and that $d < 0$ in the interior of $\Gamma$ and $d > 0$ in the exterior. The orientation of $\Gamma$ is set by taking the normal $\nu$ of $\Gamma$ to be pointing in the direction of increasing $d$ whence
\[\nu(\xi) = \nabla d(\xi),\ \xi \in \Gamma. \]
With a slight abuse of notation we also denote the projection to $\Gamma$ by $\xi$, i.e.~$\xi:U \rightarrow \Gamma$ is given by
\begin{equation}\label{eq:uniquePoint}
\xi(x) = x - d(x)\nu(x) \quad \mbox{where } \nu(x):=\nu(\xi(x)).
\end{equation}
It is worth noting that such a projection is (locally) unique provided that the width $\delta_{U} > 0$ of $U$ satisfies
\[\delta_{U} < \left[\max_{i=1,2}\norm{\kappa_{i}}_{L^{\infty}(\Gamma)}\right]^{-1}\]
where $\kappa_{i}$ denotes the $i$th principle curvature of the Weingarten map $\mathbf{H}$, given by 
\begin{align}\label{eq:WeingartenMap} 
\mathbf{H}(x) := \nabla^{2}d(x). 
\end{align} 
Later on, we will consider a triangulated surface $\Gamma_h \subset U$ approximating $\Gamma$ such that there is a one-to-one relation between points $x \in \Gamma_h$ and $\xi \in \Gamma$ so that, in particular, the above relation (\ref{eq:uniquePoint}) can be inverted. Throughout this paper, we denote by
\[ \mathbf{P}(\xi):= \mathbf{I} - \nu(\xi)\otimes \nu(\xi),\ \xi\in \Gamma, \]
the projection onto the tangent space $T_{\xi}\Gamma$ on $\Gamma$ at a point $\xi \in \Gamma$. Here $\otimes$ denotes the usual tensor product.

\begin{definition}
For any function $\eta$ defined on an open subset of $U$ containing $\Gamma$ we can define its \emph{tangential gradient} on $\Gamma$ by
\[ \nabla_{\Gamma}\eta := \nabla \eta - \left(\nabla \eta\cdot \nu \right) \nu = \mathbf{P}\nabla \eta\]
and then the \emph{Laplace-Beltrami} operator on $\Gamma$ by 
\[ \Delta_{\Gamma} \eta := \nabla_{\Gamma}\cdot (\nabla_{\Gamma} \eta).
\]
\end{definition}
\begin{definition}
We define the surface Sobolev spaces
\[ H^{m}(\Gamma) := \{u \in L^{2}(\Gamma) \ : \ D^{\alpha}u \in L^{2}(\Gamma)\ \forall |\alpha| \leq m \}, \quad m \in \mathbb{N} \cup \{ 0 \}, \]
with corresponding Sobolev seminorm and norm respectively given by
\[ 
|u|_{H^{m}(\Gamma)} := \left(\sum_{|\alpha|=m} \norm{D^{\alpha}u}_{L^{2}(\Gamma)}^{2}\right)^{1/2}, \quad 
\norm{u}_{H^{m}(\Gamma)} := \left(\sum_{k=0}^m |u|_{H^{k}(\Gamma)}^{2}\right)^{1/2}.
\]
\end{definition}     
We refer to \cite{wlokapartial} for a proper discussion of Sobolev spaces on manifolds.

Throughout this paper, we write $x \lesssim y$ to signify $x < C  y$, where $C$ is a generic positive constant whose value, possibly different at any occurrence, does not depend on the meshsize. 
Moreover, we use $x\sim y$ to state the equivalence between $x$ and $y$, i.e., $C_1 y \leq x \leq C_2 y$, for $C_1,\ C_2$ independent of the meshsize. 

\subsection{Discrete surface $\Gamma_{h}$}
The smooth surface $\Gamma$ is approximated by a polyhedral surface $\Gamma_{h} \subset U$ composed of planar triangles. Let $\mathcal{T}_{h}$ be the associated regular, conforming triangulation of $\Gamma_{h}$ i.e.
\[ \Gamma_{h} = \bigcup_{K_{h} \in \mathcal{T}_{h}} K_{h}. \]
The vertices are taken to sit on $\Gamma$ so that $\Gamma_{h}$ is its linear interpolation. We assume that the projection map $\xi$ defined in (\ref{eq:uniquePoint}) is a bijection when restricted to $\Gamma_h$, thus avoiding multiple coverings of $\Gamma$ by $\Gamma_h$.
Let $\mathcal{E}_{h}$ denote the set of all codimension one intersections of elements $K_h^{+},K_h^{-}\in \mathcal{T}_{h}$ (i.e., the edges). We define the conormal $n_h^{+}$ on such an intersection $e_{h} \in \mathcal{E}_{h}$ of elements $K_h^{+}$ and $K_h^{-}$ by demanding that \\[1mm]
$\bullet$ $n_h^{+}$ is a unit vector,\\
$\bullet$ $n_h^{+}$ is tangential to (the planar triangle) $K_h^{+}$, \\
$\bullet$ in each point $x \in e_h$ we have that $n_h^{+} \cdot (y-x) \leq 0$ for all $y \in K_h^{+}$. \\[1mm]
Analogously one can define the conormal $n_h^{-}$ on $e_h$ by exchanging $K_h^{+}$ with $K_h^{-}$. It is important to note that
\[ n_{h}^{+} \not = -n_{h}^{-} \]
 in general, and in contrast to the planar setting. 
 Finally, we will denote by $\nu_{h}$ the outward unit normal to $\Gamma_{h}$ and define for each $K_{h} \in \mathcal{T}_h$ the discrete projection $\mathbf{P}_{h}$ onto the tangential space of $\Gamma_{h}$ by
\[ \mathbf{P}_{h}(x) := \mathbf{I} - \nu_{h}(x) \otimes \nu_{h}(x),\ x \in \Gamma_{h}, \]
so that, for $v_{h}$ defined on $\Gamma_{h}$,
\[\nabla_{\Gamma_{h}}v_{h} = \mathbf{P}_{h}\nabla v_{h}.\]
Let $K \subset \mathbb{R}^{2}$ be the (flat) reference element and let $F_{K_{h}} : K \rightarrow K_{h} \subset  \mathbb{R}^{3}$ for $K_{h} \in \calT_h$. We define the DG space associated to $\Gamma_h$ by
\begin{align*}
V_{h} =\{v_{h} \in L^2(\Gamma_h): v_{h}|_{K_{h}}= \chi \circ F_{K_{h}}^{-1} ,\ \chi \in \mathbb{P}^{1}(K)\ \ \ \forall K_{h}\in \calT_{h}\}.
\end{align*}
For $v_h\in V_{h}$ we adopt the convention that $v_h^{\pm}$ is the trace of $v_h$ on $e_{h}= K_{h}^+\cap K_{h}^-$ taken within the interior of $K_{h}^{\pm}$, respectively.

\subsection{Relating $\Gamma_{h}$ to $\Gamma$}
\begin{definition}
For any function $w$ defined on $\Gamma_{h}$ we define the \emph{surface lift} onto $\Gamma$ by
\[ \eta^{l}(\xi) := \eta(x(\xi)),\ \xi \in \Gamma, \]
where by (\ref{eq:uniquePoint}) and the non-overlapping of the triangular elements, $x(\xi)$ is defined as the unique solution of
\[ x = \xi + d(x)\nu(\xi).\]
\end{definition}
Extending $\eta^{l}$ constantly along the lines $s \mapsto \xi + s\nu(\xi)$ we obtain a function defined on $U$.
By (\ref{eq:uniquePoint}), for every $K_{h} \in \mathcal{T}_{h}$, there is a unique curved triangle $K_{h}^{l} := \xi(K_{h}) \subset \Gamma$. Note that we assumed $\xi(x)$ is a bijection so multiple coverings are in fact not permitted. We now define the regular, conforming triangulation $\mathcal{T}_{h}^{l}$ of $\Gamma$ such that
\[ \Gamma = \bigcup_{K_{h}^{l} \in \mathcal{T}_{h}^{l}} K_{h}^{l}. \]
The triangulation $\mathcal{T}_{h}^{l}$ of $\Gamma$ is thus induced by the triangulation $\mathcal{T}_{h}$ of $\Gamma_{h}$ via the surface lift.
The appropriate function space for surface lifted functions is given by
\[V_{h}^{l} := \{v_{h}^{l} \in L^{2}(\Gamma)\ : \ v_{h}^{l}(\xi) =
v_{h}(x(\xi))\ \mbox{with some}\ v_{h} \in V_{h}\}.\]
We also denote by $\eta^{-l} \in V_{h}$ the \emph{inverse} surface lift of
some function $\eta \in V_{h}^{l}$, satisfying $(\eta^{-l})^{l} = \eta$.
Finally, by applying the chain rule for differentiation on (\ref{eq:uniquePoint}), one can show that for $x \in \Gamma_{h}$ and $v_{h}$ defined on $\Gamma_{h}$, we have that
\begin{equation}\label{eq:Gammah2GammaGradient} 
\nabla_{\Gamma_{h}} v_{h}(x) = \mathbf{P}_{h}(x)(\mathbf{I} - d \mathbf{H})(x) \mathbf{P}(x)\nabla_{\Gamma} v_{h}^{l}(\xi(x)).  
\end{equation}
Finally, for $x \in \Gamma_{h}$, we denote the local area  and local edge deformations when transforming $\Gamma_{h}$ to $\Gamma$ by respectively $\delta_{h}(x)$ and $\delta_{e_{h}}(x)$ i.e.
\begin{align*}
\delta_{h}(x) \dAh(x) = \dA(\xi(x)),\ \ \delta_{e_{h}}(x) \dsh(x) = \ds(\xi(x)).
\end{align*}

We finally state and prove some geometric estimates relating $\Gamma_{h}$ to $\Gamma$.
\begin{lemma} \label{Gamma2GammahSmall}
Let $\Gamma$ be a compact smooth and oriented surface in $\mathbb{R}^{3}$ and let $\Gamma_{h}$ be its linear interpolation. Then, omitting the surface lift symbols, we have that
\begin{align}
\norm{d}_{L^{\infty}(\Gamma_{h})} & \lesssim h^{2}, \label{eq:Gamma2GammahSmalla} \\
\norm{1-\delta_{h}}_{L^{\infty}(\Gamma_{h})} &\lesssim h^{2}, \label{eq:Gamma2GammahSmallb} \\
\norm{\nu-\nu_{h}}_{L^{\infty}(\Gamma_{h})}& \lesssim h, \label{eq:Gamma2GammahSmallc} \\
\norm{1-\delta_{e_{h}}}_{L^{\infty}(\mathcal{E}_{h})}& \lesssim h^{2}, \label{eq:Gamma2GammahSmalld} \\
\norm{n^{+/-} -\mathbf{P} n_{h}^{+/-}}_{L^{\infty}(\mathcal{E}_{h})} &\lesssim  h^{2} \label{eq:Gamma2GammahSmalle} \\
\norm{\mathbf{P}-\mathbf{P} \mathbf{P}_{h}}_{L^{\infty}(\Gamma_{h})} &\lesssim h \label{advrelbis} \\
\norm{\mathbf{P}-\mathbf{P} \mathbf{P}_{h} \mathbf{P}}_{L^{\infty}(\Gamma_{h})} &\lesssim h^{2} \label{advrel}. 
\end{align}
\end{lemma}
\begin{proof}
The first three estimates are proven in \cite{dziuk1988finite}. The next two estimates are proven in \cite{antonietti2013analysis}. For the last estimate, it is sufficient to show that $\left(\mathbf{P}-\mathbf{P} \mathbf{P}_{h} \mathbf{P} \right) x$ for $x \in \mathbb{R}^{3}$ scales appropriately. Setting $\tilde{x} = \mathbf{P} x$ (which is tangential to $\Gamma$) and noting that $\mathbf{P}_{h} \tilde{x} = \tilde{x} - (\tilde{x} \cdot \nu_{h})\nu_{h}$, we have that
\begin{align*}
&\left(\mathbf{P}-\mathbf{P} \mathbf{P}_{h} \mathbf{P}\right) x = \tilde{x} - \mathbf{P} \mathbf{P}_{h} \tilde{x} = \tilde{x} - \mathbf{P}\left( \tilde{x} - (\tilde{x} \cdot \nu_{h})\nu_{h} \right) = \tilde{x} - \left( \tilde{x} - (\tilde{x} \cdot \nu) \nu - (\tilde{x} \cdot \nu_{h})(\nu_{h} - (\nu_{h} \cdot \nu)\nu         \right) \\
&= (\tilde{x} \cdot \nu_{h}) \nu_{h} - (\tilde{x} \cdot \nu_{h})(\nu_{h} \cdot \nu)\nu = (\tilde{x} \cdot \nu_{h}) (\nu_{h} - \nu) + (\tilde{x} \cdot \nu_{h})\nu \left( 1 - (\nu_{h} \cdot \nu) \right) \\
&= (\tilde{x} \cdot \left( \nu_{h} - \nu \right)) (\nu_{h} - \nu) + (\tilde{x} \cdot \left(\nu_{h} - \nu \right))\nu \left( 1 - (\nu_{h} \cdot \nu) \right) \leq |\nu_{h} - \nu|^{2}|\tilde{x}| + \half |\nu_{h} - \nu|^{2}|\tilde{x}| \\
&\lesssim h^{2} |\tilde{x}| \lesssim h^{2} |x| 
\end{align*}
where we have used the equality $1 - (\nu_{h} \cdot \nu) = \half |\nu_{h} - \nu|^{2}$ and the geometric estimate (\ref{eq:Gamma2GammahSmallc}). The proof of the second to last estimate follows similar arguments.  
\end{proof}

We complete this section by defining an $L^2$ type projection operator
for function in $H^2(\Gamma_{h})$:

\begin{lemma}\label{interpolationEstimate}
Let $\eta \in H^{2}(\Gamma_{h})$  and denote by $\Pi_{hk} \eta$ the
$L^{2}$ projection of $\eta$ onto $V_{h}$ for $k=1$ and piecewise constant functions on $\Gamma_{h}$ for $k=0$. Furthermore, for notational simplicity, we define $\Pi_{h} \eta := \Pi_{h1} \eta$. 
Then, for sufficiently small $h$, we have that
\begin{align*}
\norm{\eta - \Pi_{hk} \eta}_{L^{p}(K_{h})} &\lesssim h^{k+1}\norm{\eta}_{W^{k+1,p}(K_{h})},\ \ 1 \leq p \leq \infty,\ \ k = 0, 1, \\
\norm{\eta - \Pi_h \eta}_{L^2(\partial K_h)} &\lesssim h^{3/2}\norm{\eta}_{H^{2}(K_h)}
\end{align*}
for each $K_{h} \in \mathcal{T}_{h}$.
\end{lemma}
\begin{proof}
The proof of both estimates follow from applying standard arguments on each $K_{h} \in \mathcal{T}_{h}$ (which hold since each triangle $K_{h}$ is planar).
\end{proof}


\section{Problem formulation, discretisation and properties}
\subsection{Model problem}
We will split our model problem (\ref{eq:modelproblem}) into two parts: an
elliptic part and a first-order
hyperbolic operator which, when written in weak form over $H^1(\Gamma)$, are respectively given by
\[
 \int_\Gamma \epsilon \nabla_{\Gamma}u \cdot \nabla_{\Gamma}v \dA 
\]
and
\[\int_\Gamma - w u \cdot \nabla_{\Gamma}v  + c u v \dA\]
where the velocity field $w$ can be assumed to be purely \emph{tangential} to the surface, i.e. 
$w \cdot \nu = 0$ everywhere, since any normal contribution would vanish when multiplied with 
$\nabla_{\Gamma} u$. We will also assume, for simplicity, that the velocity field is divergence-free 
which, together with $w \cdot \nu = 0$, implies that $\nabla_{\Gamma} \cdot
w = 0$. Finally the mass term is multiplied by a bounded function $c>0$. The weak problem then reads: find $u \in H^{1}(\Gamma)$ such that 
\begin{equation} \label{eq:weakH1AdvectionDiffusion}
 \int_\Gamma \left( \epsilon \nabla_{\Gamma}u - w u \right) \cdot \nabla_{\Gamma}v + c uv \dA =
   \int_\Gamma fv\dA \quad\forall v\in H^1(\Gamma). 
\end{equation}
Existence and uniqueness of a solution $u \in H^{2}(\Gamma)$ follows from standard arguments.

\subsection{Discretisation of the hyperbolic operator}
Before we define a DG discretisation, we introduce the following discrete surface trace operators:
\begin{definition}\label{def:jumpAverageOparators}
For $q \in \Pi_{K_{h} \in \mathcal{T}_{h}}L^{2}(\partial K_{h}) $, $\{q\}$ and $[q]$ are given by
\[ \{q\} := \frac{1}{2}(q^++q^-), \ [q] := q^+-q^- \ \text{on } e_{h}\in \calE_{h}. \]
For $\phi, \tilde{n} \in [\Pi_{K_{h} \in \mathcal{T}_{h}}L^{2}(\partial K_{h})]^{3}$, $\{\phi;\tilde{n}\}$ and $[\phi;\tilde{n}]$ are given by
\[ \{\phi;\tilde{n}\} := \frac{1}{2}(\phi^{+} \cdot \tilde{n}^{+} - \phi^{-} \cdot \tilde{n}^{-}),\   \ [\phi;\tilde{n}] := \phi^{+}\cdot \tilde{n}^{+}+\phi^{-}
\cdot \tilde{n}^{-}\ \ \text{on } e_{h}\in \calE_{h}. \]
\end{definition}

Now we can define the discrete bilinear form for the advection operator:
\begin{align}\label{eq:motivationEquation}
\BDG(u_{h},v_{h}) := 
  &\sum_{K_{h} \in \mathcal{T}_{h}}\int_{K_{h}}  - w_{h} u_{h} \cdot \nabla_{\Gamma_{h}} v_{h} + 
  \left( c + \gamma_{h}\right) u_{h} v_{h} \dAh +  
    \sum_{e_{h} \in \mathcal{E}_{h}}\int_{e_{h}}\widehat{w_{h} u_{h}}
    [v_{h}] \dsh + \Cw(u_h,v_h)
\end{align}      
$w_{h}$ is a discrete velocity field which will be explicitly related to
$w$ further down.
We define the surface upwind flux $\widehat{w_{h} u_{h}}$ by
\[ \widehat{w_{h} u_{h}} = \{ w_{h} ; n_{h} \}\{u_h\} +  \xi_{e_{h}}[u_{h}]  \]
where $\xi_{e_{h}} := \frac{1}{2}|\{w_{h};n_{h}\}|$.
For reasons that will be clear later on, the discrete mass perturbation coefficient $\gamma_{h}$ is given by
\begin{align}\label{eq:discreteMassCoefficient}
\gamma_{h} =
   \max\{-\nabla_{\Gamma_{h}}\cdot w_h,
         -\half (c+\nabla_{\Gamma_{h}}\cdot w_h) \}
\end{align}
Finally an additional term is added to take possible discontinuities in the discrete 
velocity field into account:
\begin{align}\label{eq:velocityJumpCorrection}
\Cw(u_h,v_h) := 
    \sum_{e_h\in\mathcal{E}_h}\int_{e_h}\frac{1}{2}[w_h;n_h]\{u_hv_h\} 
\end{align}      

We make the following assumptions on the approximation properties of the
discrete velocity field:
if $w \in [W^{2,\infty}(K_{h}^{l})]^{3}$, then for each $K_{h} \in \mathcal{T}_{h}$
we have
\begin{align}
\label{eq:advectionGeometricEstimate1a}
&\norm{\mathbf{P}_{h}w^{-l} - w_{h}}_{L^{\infty}(K_{h})} \lesssim h^{p_1},\\
\label{eq:advectionGeometricEstimate2}
&\norm{\nabla_{\Gamma_{h}} \cdot w_{h}}_{L^{\infty}(K_{h})} \lesssim h^{p_2} 
\end{align}
where $p_1,p_2 > 0$. Note that assumption (\ref{eq:advectionGeometricEstimate2}) implies that
\begin{align}\label{eq:eq:advectionGeometricEstimate2Corollary}
\gamma_{h} \equiv -\nabla_{\Gamma_h}w_h\ \ \mbox{for}\ h\ \mbox{small enough}.
\end{align}  

\begin{remark}
Note that
\[
\widehat{w_{h} u_{h}}[v_h] =
\left\{
	\begin{array}{lll}
    u_h^+(w_h^+\cdot n_h^+ - w_h^-\cdot n_h^-)(v_h^+-v_h^-) 
		      & \mbox{if } \{w_{h};n_h\} \cdot n_{h}^{+} > 0; \\
    u_h^-(w_h^+\cdot n_h^+ - w_h^-\cdot n_h^-)(v_h^+-v_h^-) 
		      & \mbox{if } \{w_{h};n_h\} \cdot n_{h}^{+} < 0; \\
    0 & \mbox{if } w_{h} \cdot n_{h}^{+} = 0, \\
	\end{array}
\right.
\]
and it can thus be seen that this flux works in exactly the same way as the classical (planar) upwind flux
in the case that $w_h^+\cdot n_h^+ - w_h^-\cdot n_h^- = w\cdot n_h^+$.
\end{remark}

We will discuss options for defining a discrete velocity field after
proving stability and error estimates for the discretization.

\subsection{Stability of the advection operator}
In the following we shall prove stability of $\BDG$ in the norm 
\begin{align}\label{def:DGNorm} 
\hbarnorm{\cdot}^{2} := 
\norm{\cdot}_{L^{2}(\Gamma_{h})}^{2} 
+ \sum_{e_{h} \in \mathcal{E}_{h}} \left\norm{\sqrt{\xi_{e_{h}}} [ \cdot ]\right}_{L^{2}(e_{h})}^{2}. 
\end{align}
We first require a useful formula which holds for functions in
\[ H^{1}(\mathcal{T}_{h}) := \{ v|_{K_{h}} \in H^{1}(K_{h})\ :\ \forall K_{h} \in \mathcal{T}_{h} \}. \]
\begin{lemma}\label{eq:intByParts}
Let $\phi \in [H^{1}(\mathcal{T}_{h})]^{3}$ and $\psi \in H^{1}(\mathcal{T}_{h})$. Then we have that
\begin{align*}
\sum_{K_{h} \in \calT_{h}}\int_{\partial K_{h}}\psi \phi \cdot n_{K_{h}} \dsh =\sum_{e_{h} \in \calE_{h}}\int_{e_{h}}[\phi;n_{h}]\{\psi\} + \{\phi;n_{h}\}[\psi] \dsh.
\end{align*}
\end{lemma}
\begin{proof}
The result follows straightforwardly by noting that
\begin{align*}
\sum_{K_{h} \in \calT_{h}}\int_{\partial K_{h}}\psi \phi \cdot n_{K_{h}} \dsh &=  \sum_{e_{h} \in \calE_{h}}\int_{e_{h}}[\psi \phi; n_{h}] \dsh = \sum_{e_{h} \in \calE_{h}}\int_{e_{h}}[\phi;n_{h}]\{\psi\} + \{\phi;n_{h}\}[\psi] \dsh.
\end{align*}
\end{proof}
\begin{lemma}\label{Lemma:boundStabGammakIPAdvection}
The surface DG bilinear form $\BDG$ is stable, i.e.,
\begin{equation*}
\begin{aligned}
&\BDG(u_h,u_h)\gtrsim  \hbarnorm{u_h}^2,
\end{aligned}
\end{equation*}
for every $u_h,v_h \in V_{h}$, provided that the discrete velocity field $w_{h}$ satisfies 
(\ref{eq:advectionGeometricEstimate2}).
\end{lemma}
\begin{proof}
We proceed along the lines of \cite{brezzi2004discontinuous} by testing (\ref{eq:motivationEquation}) 
with $v_{h} = u_{h}$ and integrating by parts on each $K_{h} \in
\mathcal{T}_{h}$, applying Lemma \ref{eq:intByParts}.
By doing so, we obtain
\begin{align*}
\BDG(u_{h},u_{h}) &=
  \sum_{K_{h} \in \mathcal{T}_{h}}\int_{K_{h}}\left( c + \gamma_{h} + 
         \half \nabla_{\Gamma_{h}} \cdot w_{h} \right) u_{h}^{2} \dAh + 
         \sum_{e_{h} \in \mathcal{E}_{h}}\int_{e_{h}}  
                        \{ w_{h};n_h\}\{ u_{h} \}[u_{h}] + 
                        \xi_{e_{h}}|[u_{h}]|^{2} \dsh \notag \\
  &\qquad - \half \sum_{e_{h} \in \mathcal{E}_{h}}\int_{e_{h}}
      \{w_{h};n_{h}\}[u_{h}^{2}] + [w_{h};n_{h}]\{u_{h}^{2}\}\dsh
   + \half \sum_{e_{h} \in \mathcal{E}_{h}}\int_{e_{h}} [w_h;n_h]\{u_h^2\} \notag
\\
  &= 
  \sum_{K_{h} \in \mathcal{T}_{h}}\int_{K_{h}}\left( c + \gamma_{h} + 
         \half \nabla_{\Gamma_{h}} \cdot w_{h} \right) u_{h}^{2} \dAh \\
  &\qquad + \sum_{e_{h} \in \mathcal{E}_{h}}\int_{e_{h}}  
    \{ w_{h}; n_{h} \}\{u_h\}[u_{h}] + \xi_{e_{h}}|[u_{h}]|^{2} 
    - \half \{w_{h};n_{h}\}[u_{h}^{2}] 
\\
  &= 
  \sum_{K_{h} \in \mathcal{T}_{h}}\int_{K_{h}}\left( c + \gamma_{h} + 
         \half \nabla_{\Gamma_{h}} \cdot w_{h} \right) u_{h}^{2} \dAh \\
  &\qquad + \sum_{e_{h} \in \mathcal{E}_{h}}\int_{e_{h}}  
    \{ w_{h}; n_{h} \}\{u_h\}[u_{h}] 
    + \xi_{e_{h}}|[u_{h}]|^{2} 
    - \{w_{h};n_{h}\}[u_h]\{u_h\} \\
  &= 
  \sum_{K_{h} \in \mathcal{T}_{h}}\int_{K_{h}}
    \left( c + \gamma_{h} + 
         \half \nabla_{\Gamma_{h}} \cdot w_{h} \right) 
    u_{h}^{2} \dAh 
   + \sum_{e_{h} \in \mathcal{E}_{h}}\int_{e_{h}}  \xi_{e_{h}}|[u_{h}]|^{2} 
  \gtrsim  \hbarnorm{u_h}^2.
\end{align*}
In the final estimate we used that
$    c + \gamma_{h} + 
         \half \nabla_{\Gamma_{h}} \cdot w_{h} 
      \geq \half c > 0 $
due to our definition of $\gamma_{h}$ given
in (\ref{eq:discreteMassCoefficient}).
\end{proof} 

\subsection{Discretisation of the advection-diffusion operator}
We consider a bilinear form 
\begin{align}\label{eq:InteriorPenaltyGammahFormAdvectionDiffusion}
\CDG(u_{h},v_{h}) := \ADG(u_{h},v_{h}) + \BDG(u_{h},v_{h})
\end{align}
to discretize our problem \eqref{eq:modelproblem}. Here $\ADG(u_{h},v_{h})$ is a
discretization of $-\epsilon\triangle u$ on $\Gamma_h$. 
A number of formulations are given in Section 3.2 of \cite{antonietti2013analysis}.
One such example is the symmetric surface interior penalty (IP) method:
\begin{align}\label{eq:InteriorPenaltyGammahForm}
\ADG(u_{h},v_{h}) &= \sum_{\WH{K}_{h} \in \WH{\calT}_{h}}\int_{\WH{K}_{h}}\epsilon\nabla_{\Gamma_{h}}u_{h}\cdot \nabla_{\Gamma_{h}}v_{h} \dAh + \sum_{\WH{e}_{h} \in \WH{\calE}_{h}}\int_{\WH{e}_{h}} \beta_{\WH{e}_{h}}[u_{h}][v_{h}]\ \dsh \notag \\
& - \sum_{\WH{e}_{h} \in \WH{\calE}_{h}}\int_{\WH{e}_{h}} \big( [u_{h}] \{\epsilon \nabla_{\Gamma_{h}}v_{h}; \WH{n}_{h} \} + [v_{h}]\{\epsilon \nabla_{\Gamma_{h}}u_{h}; \WH{n}_{h}\} \big) \dsh
\end{align}
where $\beta_{\WH{e}_{h}} := \frac{\epsilon \alpha}{h}$ with $\alpha > 0$
being a (penalty) parameter at our disposition. 
This method was first considered in \cite{dedner2012analysis} for the 
elliptic problem including a mass term of the form $\epsilon uv$. The
method was shown to be 
stable for functions in $V_{h}$ in the norm given by
\begin{align} \label{eq:ellipticnormh}
\norm{\cdot}_{D}^{2} := \epsilon \sum_{K_{h} \in \mathcal{T}_{h}} \norm{\cdot}_{H^{1}(K_{h})}^{2} + 
     \sum_{e_{h} \in \mathcal{E}_{h}} \left\norm{\sqrt{\beta_{e_{h}}} [ \cdot ]\right}_{L^{2}(e_{h})}^{2}, 
\end{align}
provided that $\alpha$ is large enough. With this in mind and taking Lemma
\ref{Lemma:boundStabGammakIPAdvection} into account, we can easily prove the
stability of $\CDG$ in the following norm:
\begin{align} 
\barnorm{\cdot}^{2} &:= \norm{\cdot}_{D}^2 + \hbarnorm{\cdot}^2 \\
&=
     \norm{\cdot}^2_{L^2(\Gamma_h)} +
     \epsilon \sum_{K_{h} \in \mathcal{T}_{h}} \norm{\cdot}_{H^{1}(K_{h})}^{2} + 
     \sum_{e_{h} \in \mathcal{E}_{h}}\left(  
     \left\norm{\sqrt{\beta_{e_{h}}} [ \cdot ]\right}_{L^{2}(e_{h})}^{2} +
     \left\norm{\sqrt{\xi_{e_{h}}} [ \cdot ]\right}_{L^{2}(e_{h})}^{2}
     \right). \nonumber
\end{align}

\subsection{Bilinear form on $\Gamma$}
We end this section by defining a bilinear form on $\Gamma$ induced by $\BDG$:
\begin{align}\label{eq:motivationEquationGamma}
\BC(v_{1},v_{2}) = &\sum_{K_{h}^{l} \in
\mathcal{T}_{h}^{l}}\int_{K_{h}^{l}}  - w v_{1}  \cdot \nabla_{\Gamma}
v_{2}  + cv_{1} v_{2}\dA +  \sum_{e_{h}^{l} \in \mathcal{E}_{h}^{l}}
\int_{e_{h}^{l}}\left( \{w; n\}\{v_1\} + \xi_{e_{h}^{l}} [v_{1}] \right) [v_{2}] \ds. 
\end{align}
where $\xi_{e_{h}^{l}} := \delta_{e_h}^{-1}\frac{1}{2}|w \cdot n|$ with 
$n=n^{+}$ or $n=n^{-}$ is a surface conormals to $e_h^l$.
Note that since $w$ is assumed to be continuous on $\Gamma$ and
$n^{+}=-n^{-}$ we have $[w;n]=0$ and thus the addition term $\Cw$ vanishes
on $\Gamma$ and $\{w;n\} = w\cdot n^+ = -w\cdot n^-$.
Stability of $\BC(\cdot, \cdot)$ in $V_{h}^{l} \times V_{h}^{l}$ follows from similar arguments as the proof of Lemma \ref{Lemma:boundStabGammakIPAdvection}. The norm associated with the bilinear form $\BC$
is given by
\begin{align}\label{def:DGNormGamma} 
\hbarnorm{\cdot}^{2} := \norm{\cdot}_{L^{2}(\Gamma)}^{2} + 
    \sum_{e_{h}^{l} \in \mathcal{E}_{h}^{l}} \left\norm{\sqrt{\xi_{e_{h}^{l}}}[ \cdot ]\right}_{L^{2}(e_{h}^{l})}^{2}. 
\end{align}
Note the slight abuse of notation since here and in the following we use the same symbols 
to denote norms on $\Gamma_h$ and on $\Gamma$. It will always be clear from the
argument which norm we are refering too.

Finally, we define
\begin{align}\label{eq:InteriorPenaltyGammaForm}
\AC(v_{1},v_{2}) & =  \sum_{\WH{K}_{h}^{\ell} \in \WH{\calT}_{h}^{\ell}}\int_{\WH{K}_{h}^{\ell}}\epsilon \nabla_{\Gamma}v_{1}\cdot \nabla_{\Gamma}v_{2} \dA 
  - \sum_{\WH{e}_{h}^{\ell} \in \WH{\calE}_{h}^{\ell}}\int_{\WH{e}_{h}^{\ell}}[v_{1}]\{\epsilon \nabla_{\Gamma}v_{2}; n \} 
       + [v_{2}]\{\epsilon \nabla_{\Gamma}v_{1}; n \} \ds \notag \\
&+ \sum_{\WH{e}_{h}^{\ell} \in \WH{\calE}_{h}^{\ell}}\int_{\WH{e}_{h}^{\ell}}\beta_{\WH{e}^{l}_{h}}[v_{1}][v_{2}] \ds.
\end{align}
with $\beta_{e^{l}_{h}}:=\delta_{e_h}^{-1}\beta_{e_h}$.
Note that $\AC$ is stable in $V_{h}^{l} \times V_{h}^{l}$ with respect to the DG norm given by
\begin{align} \label{eq:ellipticnorm}
\norm{\cdot}_{D}^{2} := \epsilon \sum_{K_{h}^{l} \in \mathcal{T}_{h}^{l}} 
  \norm{\cdot}_{H^{1}(K^{l}_{h})}^{2} + 
   \sum_{e^{l}_{h} \in \mathcal{E}_{h}^{l}}
   \left\norm{\sqrt{\beta_{e^{l}_{h}}} [ \cdot ]\right}_{L^{2}(e^{l}_{h})}^{2}.
\end{align}
Furthermore it is bounded in $H^2(\Gamma)+V_h^l\times V_h^l$.
See \cite{dedner2012analysis} for further details.

Note that $u$ satisfies
\begin{equation} \label{eq:InteriorPenaltyGammaAdvection}
\CDGl(u,v):=\AC(u,v) + \BC(u,v) 
= \sum_{K_{h}^{l} \in \mathcal{T}_{h}^{l}}\int_{K_{h}^{l}}f v\ dA\
\ \forall v \in H^{2}(\Gamma) + V_{h}^{l}.
\end{equation}


\section{Convergence} \label{sec:ConvergenceAdvection}
A DG discretisation of (\ref{eq:weakH1AdvectionDiffusion}) is given as
follows: find $u_{h} \in \WH{V}_{h}$  such that
\begin{equation} \label{eq:InteriorPenaltyGammahAdvection}
\CDG(u_{h}, v_{h}) = \sum_{\WH{K}_{h} \in
\mathcal{\WH{T}}_{h}}\int_{\WH{K}_{h}}f_{h} v_{h} \dAh\ \forall v_{h} \in
\WH{V}_{h}
\end{equation}
We will consider the error of the discretization in the norm 
\begin{align*}
  \barnorm{\cdot}^2 := \hbarnorm{\cdot}^2+ \norm{\cdot}_{D}^{2} 
\end{align*}
on $\Gamma$. Before we state the main result, we make a note of key estimates relating
norms defined on $\Gamma$ to those on $\Gamma_h$:
\begin{lemma}\label{normEquivalence}
Let $v_h\in V_h$, then we have the following equivalence result:
\[ \barnorm{v_h} \lesssim \barnorm{v_h^l} \lesssim \barnorm{v_h}. \]
Furthermore, for $u\in H^2(\Gamma)$, we have that
\[ \|u^{-l}\|_{H^2(\Gamma_h)} \lesssim \|u\|_{H^2(\Gamma)}. \]
\end{lemma}
\begin{proof}
The proof of the first estimate follows similar arguments to that of Lemma 3.3 in \cite{dedner2012analysis}. The second estimate is given in (2.17) in \cite{demlow2009higher}.
\end{proof}

\begin{theorem}\label{aprioriErrorEstimateAdvectionDiffusionIP}
Let $u \in H^{2}(\Gamma)$ and $u_{h} \in V_{h}$ denote the solutions to
(\ref{eq:weakH1AdvectionDiffusion}) and (\ref{eq:motivationEquation}), respectively.  
Then, under assumptions (\ref{eq:advectionGeometricEstimate1a}) 
and (\ref{eq:advectionGeometricEstimate2}) on the discrete velocity field $w_{h}$, 
we have for $h$ small enough:
\[ \barnorm{u-u_{h}^{l}} \lesssim 
  \left( h\epsilon^{1/2} + h^{3/2} + h^{p_1} +
       s_w(h+h^{p_1-1}) \right)
       \norm{u}_{H^{2}(\Gamma)} 
\]
The constant $s_w$ is equal to $0$ if $[w_h;n_h]=0$
and equal to $1$ otherwise.
\end{theorem}
We first give an outline of the proof:
We split the error into two parts 
$\eta = u^{-l} - \Pi_{h}u^{-l}$ and $\chi = \Pi_{h}u^{-l}-u_h$
in the spirit of \cite{brezzi2004discontinuous}:
\begin{align} \label{eq:first_splitting}
\barnorm{u-u_h^{l}} \lesssim \barnorm{u^{-l}-u_h} 
  \lesssim \barnorm{\eta}+\barnorm{\chi} 
\end{align}
The first term is a projection error and can be bounded using
Lemma~\ref{interpolationEstimate} and noting that
$\barnorm{\eta}\lesssim \barnorm{\eta^l}$:
\begin{lemma}\label{BarInterpolation}
Let $\eta$ be given as before. The we have that
\[ \barnorm{\eta} \lesssim h(\epsilon^{\frac{1}{2}}+h^{\frac{1}{2}})\|u\|_{H^2(\Gamma)}
\]
for $h$ small enough.
\end{lemma}

Using the stability of $\ADG(\cdot,\cdot)$ and $\BDG(\cdot, \cdot)$ in 
$V_{h} \times V_{h}$, the second term in~\eqref{eq:first_splitting} can be
estimated by
\begin{align} \label{eq:second_splitting}
\barnorm{\chi}^2
  \lesssim \CDG(u^{-l}-u_{h},\chi) -  \CDG(\eta,\chi),
\end{align}
Note that since we do not directly have Galerkin orthogonality, the first term on the right-hand side of 
(\ref{eq:second_splitting}) is not zero. We will discuss this term at the
end of this section.
We can deal with the second term in~\eqref{eq:second_splitting} using the following lemma:
\begin{lemma}\label{ApproxErrorIterpLemma}
With $\eta$ and $\chi$ defined as before, we have that
\begin{align}
\CDG(\eta, \chi)&\lesssim  h(\epsilon^{1/2}+h^{1/2}) \norm{u}_{H^{2}(\Gamma)} \barnorm{\chi} 
\end{align}
for $h$ small enough.
\end{lemma}
\begin{proof}
The proof is a direct extension of the corresponding result in \cite{brezzi2004discontinuous} to triangulated surfaces.
We will first show that
\begin{align*}
  \BDG(\eta,\chi) \lesssim h^{3/2}\norm{u}_{H^{2}(\Gamma)}\barnorm{\chi}~.
\end{align*}
Since $\nabla \chi \in [V_{h}]^{3}$ and $\mathbf{P}_{h}$ is constant on each $K_{h} \in \mathcal{T}_{h}$, we have by definition that
\begin{align*}
\int_{K_{h}} \left( \Pi_{h0}w_{h} \cdot \nabla_{\Gamma_{h}} \chi
\right) \eta \dAh = \int_{K_{h}} \left(\mathbf{P}_{h} \Pi_{h0}w_{h}
\cdot \nabla_{\Gamma_h} \chi \right) \eta \dAh = 0. 
\end{align*}
Using this, together with (\ref{eq:eq:advectionGeometricEstimate2Corollary}), (\ref{eq:advectionGeometricEstimate2}), an inverse inequality on $\Gamma_{h}$ and Lemma \ref{interpolationEstimate}, the
element integral term of $\BDG$ with $v_{1} = \eta$ and $v_{2} = \chi$ becomes, for $h$ small enough,
\begin{align*}
&\sum_{K_{h} \in \mathcal{T}_{h}}\int_{K_{h}}  - \eta  \left( w_{h}
\cdot \nabla_{\Gamma_{h}} \chi \right) + (c + \gamma_{h})\chi \eta \dAh = \sum_{K_{h}^{l} \in
\mathcal{T}_{h}}\int_{K_{h}} \left( \left( 
\Pi_{h0} w_{h} - w_{h}\right) \cdot \nabla_{\Gamma_{h}} \chi \right)\eta + (c - \nabla_{\Gamma_{h}} \cdot w_{h}) \chi \eta \dAh \\
&\lesssim \left( \sum_{K_{h} \in \mathcal{T}_{h}} \norm{\Pi_{h0} w_{h} -
w_{h}}_{L^{\infty}(K_{h})}|\chi|_{H^{1}(K_{h})} + \norm{\chi}_{L^{2}(K_{h})} \right) \norm{\eta}_{L^{2}(K_{h})} \\
&\lesssim \left( \sum_{K_{h} \in \mathcal{T}_{h}}  h
\norm{w_{h}}_{W^{1,\infty}(K_{h})} h^{-1} \norm{\chi}_{L^{2}(K_{h})}  +
\norm{\chi}_{L^{2}(K_{h})} \right) \norm{\eta}_{L^{2}(K_{h})}
\lesssim h^{2} \norm{u}_{H^{2}(\Gamma)} \norm{\chi}_{L^{2}(\Gamma_{h})}
\end{align*}
where, in the last estimate, we have made use of the second estimate in Lemma \ref{normEquivalence}. For the edge integrals in $\BDG$ involving the upwind flux, we first observe that
\[ | \{w_{h} ; n_{h} \}\{\eta\} | = |\{w_{h} ; n_{h} \}| | \{\eta \} | \lesssim \xi_{e_{h}} | \{\eta \} |. \]
We then have that
\[\int_{e_{h}} \{w_{h} ; n_{h}\}\{\eta\}[\chi]\dsh \lesssim
\left\norm{\sqrt{\xi_{e_{h}}}\{\eta\}\right}_{L^{2}(e_{h})}
\left\norm{\sqrt{\xi_{e_{h}}}[\chi]\right}_{L^{2}(e_{h})}.   \]
Combining this with
\[ \int_{e_{h}} \xi_{e_{h}} [\eta] [\chi] \dsh \leq
\left\norm{\sqrt{\xi_{e_{h}}}[\eta]\right}_{L^{2}(e_{h})}
\left\norm{\sqrt{\xi_{e_{h}}}[\chi]\right}_{L^{2}(e_{h})}~,  \]
we may bound the flux integral term of $\BDG$ as follows:
\begin{align*}
\sum_{e_{h} \in \mathcal{E}_{h}}\int_{e_{h}}\left( 
\{w_{h}; n_{h}\}\{\eta\} + \xi_{e_{h}} [\eta] \right) [\chi] \dsh &\leq \sum_{e_{h} \in
\mathcal{E}_{h}} \left(
\left\norm{\sqrt{\xi_{e_{h}}}\{\eta\}\right}_{L^{2}(e_{h})} +
\left\norm{\sqrt{\xi_{e_{h}}}[\eta]\right}_{L^{2}(e_{h})}  \right)
\left\norm{\sqrt{\xi_{e_{h}}}[\chi]\right}_{L^{2}(e_{h})} \\
&\lesssim h^{3/2} \norm{u}_{H^{2}(\Gamma)} \left(
\left\norm{\sqrt{\xi_{e_{h}}}[\chi]\right}_{L^{2}(e_{h})} \right)^{1/2}
\end{align*}
where, again, we have made use of the second estimate in Lemma \ref{normEquivalence} to obtain the last estimate. To obtain estimates for 
$\Cw(\cdot,\cdot)$, we first note that
\[[w_{h};n_{h}] = [\mathbf{P}_{h}w_{h} - w^{-l} ; n_{h}] + [w^{-l} ; n_{h}] = [\mathbf{P}_{h}w_{h} - w^{-l} ; n_{h}] + [w^{-l} ; \mathbf{P}^{-l} n_{h} - n^{-l}].\]
And so, by assumption (\ref{eq:advectionGeometricEstimate2}) and the geometric estimate (\ref{eq:Gamma2GammahSmalle}), we have that
\[ \norm{[w_{h};n_{h}]}_{L^{\infty}(\mathcal{E}_{h})} \lesssim h^{p_{1}} + h^{2}. \]
Hence, 
\begin{align*}
\sum_{e_{h} \in \mathcal{E}_{h}}\int_{e_h}\frac{1}{2}[w_h;n_h]\{\eta \chi\} &\lesssim \norm{[w_{h};n_{h}]}_{L^{\infty}(\mathcal{E}_{h})} \sum_{e_{h} \in \mathcal{E}_{h}} \norm{\{\eta\}}_{L^{2}(e_{h})}
\norm{\{\chi\}}_{L^{2}(e_{h})} \\ 
&\lesssim (h^{p_{1}} + h^{2}) h^{3/2} \norm{u}_{H^{2}(\Gamma)} h^{-1/2}\norm{\chi}_{L^{2}(\Gamma_{h})} \\ 
&\lesssim (h^{p_{1}+1} + h^{3}) \norm{u}_{H^{2}(\Gamma)} \norm{\chi}_{L^{2}(\Gamma_{h})}.
\end{align*} 
The first two terms of $\ADG$ (given in \eqref{eq:InteriorPenaltyGammahForm}) can be easily shown to scale like the desired final estimate by applying the projection estimate given in Lemma \ref{BarInterpolation} and making use of the second estimate in Lemma \ref{normEquivalence}, as before. For the last term in \eqref{eq:InteriorPenaltyGammahForm}, we will require the following inverse estimate, adapted from Lemma 4.4 in \cite{antonietti2013analysis}: 
\[\norm{\nabla_{\Gamma_h}\chi}_{L^2(\partial \WH{K}_{h})}\lesssim  h^{-1/2}\norm{\nabla_{\Gamma_h} \chi}_{L^2(\WH{K}_{h})}.\]
Making use of this estimate, we have that
\begin{align*}
\sum_{\WH{e}_{h}\in\WH{\calE}_{h}}\int_{\WH{e}_h}[\eta] \{\epsilon \nabla_{\Gamma_h} \chi; n_h\} \dsh &\leq \sum_{\WH{e}_{h}\in\WH{\calE}_{h}}\left\norm{\beta_{\WH{e}_{h}}^{1/2}[\eta]\right}_{L^2(\WH{e}_{h})}\left\norm{\beta_{\WH{e}_{h}}^{-1/2}\{\epsilon \nabla_{\Gamma_h} \chi; n_h\}\right}_{L^2(\WH{e}_{h})}\notag \\
&\lesssim \sum_{\WH{K}_{h}\in\WH{\calT}_h} \left\norm{\beta_{\WH{e}_{h}}^{1/2}[\eta]\right}_{L^2(\partial K_{h})}\norm{\epsilon^{1/2}\nabla_{\Gamma_h} \chi}_{L^2(\WH{K}_{h})} \\
&\lesssim h \epsilon^{1/2} \norm{u}_{H^{2}(\Gamma)} \barnorm{\chi}. 
\end{align*}
\end{proof}

For the first term on the right-hand side of (\ref{eq:second_splitting}), we require the following \emph{perturbed} Galerkin orthogonality result:
\begin{lemma} \label{PerturbedGalerkinOrthogonalityAdvection}
Let $u \in H^s(\Gamma)$, $s \geq 2$, and $u_{h} \in V_{h}$ denote the
solutions to 
\eqref{eq:InteriorPenaltyGammahAdvection} and (\ref{eq:motivationEquation}), respectively. We define
the functional $E_{h}$ on $V^{l}_{h}$ by
\begin{equation*} 
E_{h}(v_h)  := \CDG(u^{-l}-u_{h},v_{h}) = I_D+I_B+I_f
\end{equation*}
where the three contributions $E_D,E_B$, and $E_f$ meassure the variational
crime from the elliptic, the hyperbolic, and the right hand side descritization,
respecively.

Then if $w_h$ satisfies
(\ref{eq:advectionGeometricEstimate1a})
and (\ref{eq:advectionGeometricEstimate2}), we have that
\begin{align} \label{eq:Eh_quadAdvData}
  |I_D(v_h)|+|I_f(v_h)| &\lesssim
  h^{2}\norm{f}_{L^2(\Gamma)}\barnorm{v_h}~,\\
  |I_B(v_h)| &\lesssim \big(h^{2}+h^{p_1}+s_wh^{-1}(h^2+h^{p_1})\big)\norm{u}_{H^2(\Gamma)}\hbarnorm{v_h},
\end{align}
for $h$ small enough. The constant $s_w$ is defined in Theorem
\ref{PerturbedGalerkinOrthogonalityAdvection}.
\end{lemma}
Before we give its full proof, we will complete that of
Theorem~\ref{aprioriErrorEstimateAdvectionDiffusionIP} 
assuming this result. Starting again with the splitting
(\ref{eq:first_splitting}) and \eqref{eq:second_splitting}:
\begin{align*}
\barnorm{u-u_h^{l}} 
 &\lesssim \barnorm{u^{l}-u_h}
  \lesssim \barnorm{\eta} + 
  \left(\CDG(u-u_{h}^{l},\chi) - \CDG(\eta,\chi)\right)\barnorm{\chi}^{-1} \\
  &\lesssim \left(h\epsilon^{\frac{1}{2}} + h^{\frac{3}{2}}\right)\|u\|_{H^2}
        + \left(h^2+\epsilon^{-\frac{1}{2}}h^2\right)\norm{f}_{L^2}
        + \left(h\epsilon^{\frac{1}{2}}+h^{\frac{3}{2}}\right)\|u\|_{H^2}
\end{align*}
Which completes the proof of our main Theorem.


\begin{proof}[Proof of Lemma \ref{PerturbedGalerkinOrthogonalityAdvection}]
Using the definition of $u_h$ and the fact that $u$ solves the variational
problem \eqref{eq:weakH1AdvectionDiffusion} we have 
\begin{align*}
E_{h}(v_h) &= \CDG(u^{-l}-u_{h},v_{h}) 
            = \CDG(u^{l},v_h) - \int_{\Gamma_h} f_hv_h + \int_{\Gamma} fv_h^l
               - \CDGl(u,v_h^l) \\
           &= \ADG(u^{l},v_h) - \AC(u,v_h^l) + \BDG(u^l,v_h) - \BC(u.v_h^l) 
              + \int_{\Gamma} fv^l_h - \int_{\Gamma_h} f_hv_h
            = I_D+I_B+I_f
\end{align*}
For $I_f$ we can simply write using that $f_h=f^{-l}_h$:
\begin{align*}
  I_f &= \int_{\Gamma} \big(1-\delta_h^{-1}\big)fv^l_h 
      \lesssim h^2\norm{f}_{L^2(\Gamma)}
\end{align*}
The error coming from the diffusion part was throughly studied in
\cite{dedner2012analysis} where it was shown to scale satisfy the desired
bound.
Finally, we rewrite the error term comming from the advection discretization:
\begin{align*}
  I_B &= \sum_{K_{h} \in \mathcal{T}_{h}}\int_{K_{h}}  - w_{h} u^{-l} \cdot \nabla_{\Gamma_{h}} v_{h} + 
           \left( c + \gamma_{h}\right) u^{-l} v_{h} \dAh +  
         \sum_{e_{h} \in \mathcal{E}_{h}}\int_{e_{h}}
            \widehat{w_{h} u^{-l}} [v_{h}] +
            \frac{1}{2}[w_h;n_h]u^{-l}{v_h}\dsh\\
      &\qquad +
               \Cw(u^{-l},v_h)  - 
         \sum_{K_{h} \in \mathcal{T}_{h}}\int_{K^l_{h}}  - w u \cdot \nabla_{\Gamma} v^l_{h} + 
           c u v^l_{h} \dA -
         \sum_{e_{h} \in \mathcal{E}_{h}}\int_{e^l_{h}}\widehat{w u}[v^l_{h}] \ds
\end{align*}
Noting that $u,u^{-l}$ are continuous functions the numerical fluxes reduce
to
\[ 
  \widehat{w_{h} u^{-l}} = \{ w_{h} ; n_{h} \}u^{-l}~, \qquad 
     \widehat{w u} = \{ w ; n \}u
\]
Next we use the integration by parts formula 
\eqref{eq:weakH1AdvectionDiffusion} 
again taking into
account that $[u]=[u^{-l}]=0$, arriving at:
\begin{align*}
  I_B &= \sum_{K_{h} \in \mathcal{T}_{h}}\int_{K_{h}}  w_{h}
         v_{h} \cdot \nabla_{\Gamma_{h}} u^{-l} + 
         \left( c + \gamma_{h}+\nabla_{\Gamma_h}\cdot w_h\right) u^{-l} v_{h} \dAh 
         - 
         \sum_{e_h\in\mathcal{E}_h}\int_{e_h} \frac{1}{2}[w_h;n_h]u^{-l}\{v_h\}  \\
      &\qquad-
         \sum_{K_{h} \in \mathcal{T}_{h}}\int_{K^l_{h}}  w
         v^l_{h} \cdot \nabla_{\Gamma} u + 
         c u v^l_{h} \dA 
\end{align*}
Note that for $h$ small enough $\gamma_{h}+\nabla_{\Gamma_h}\cdot w_h=0$
based on the definition \eqref{eq:discreteMassCoefficient}
and using that $\nabla_{\Gamma_h}\cdot u_h\to 0$.
We now lift the volume integrals on $\Gamma_h$ to $\Gamma$
\begin{align*}
  I_B &= \sum_{K_{h} \in \mathcal{T}_{h}}\int_{K^l_{h}}  \delta_h^{-l}w^l_{h}
         v^l_{h} \cdot P_h(I-dH)P\nabla_{\Gamma} u + 
         c u v^l_{h} \dA  
       - \sum_{e_h\in\mathcal{E}_h}\int_{e^l_h} \frac{1}{2}[w_h;n_h]u^{-l}\{v_h\} \\ 
      &\qquad-
         \sum_{K_{h} \in \mathcal{T}_{h}}\int_{K^l_{h}}  w
         v^l_{h} \cdot \nabla_{\Gamma} u + 
         c u v^l_{h} \dA \\
      &= \sum_{K_{h} \in \mathcal{T}_{h}}\int_{K^l_{h}}  (\delta_h-1)
         \big( w^l_{h} v^l_{h} \cdot P_h(I-dH)P\nabla_{\Gamma} u \dA \\
      &\qquad+ \sum_{K_{h} \in \mathcal{T}_{h}}\int_{K^l_{h}} 
         \big(w^l_{h}-w\big) v^l_{h} \cdot P_h(I-dH)P\nabla_{\Gamma} u \dA \\
      &\qquad+ \sum_{K_{h} \in \mathcal{T}_{h}}\int_{K^l_{h}} 
               \big( w v^l_{h} \cdot (P_h(I-dH)P-I)\nabla_{\Gamma} u \dA \\
      &\qquad
      - \sum_{e_h\in\mathcal{E}_h}\int_{e^l_h} \frac{1}{2}[w_h;n_h]u^{-l}\{v_h\}  
      = I_B^1+I_B^2+I_B^3+I_B^4
\end{align*}
We will bound each of these terms using results from
Lemma~\eqref{Gamma2GammahSmall} and Assumptions
(\ref{eq:advectionGeometricEstimate1a})
and (\ref{eq:advectionGeometricEstimate2}):
\begin{align*}
I_B^1 &\lesssim
        h^2\norm{w_h^l}_{L^\infty(\Gamma)}\norm{P_h(I-dH)P}_{L^\infty(\Gamma)} 
          \sum_{K_{h} \in \mathcal{T}_{h}}\int_{K^l_{h}} 
          \norm{v_h^l}_{L^2(K^l_h)}\norm{\nabla u}_{L^2(K_h^l)} 
      \lesssim h^2\hbarnorm{v_h^l}\norm{u}_{H^2(\Gamma)}
\end{align*}
\begin{align*}
I_B^2 &\lesssim 
        \sum_{K_{h} \in \mathcal{T}_{h}}\int_{K^l_{h}} 
          \norm{v^l_h}_{L^2(K^l_h)} 
          \norm{(w^l_{h}-w) \cdot P_h(I-dH)P\nabla_{\Gamma} u}
      \lesssim h^{p_1}\hbarnorm{v_h^l}\norm{u}_{H^2(\Gamma)}
\end{align*}
For the next estimate we use that $w$ is tangent to $\Gamma$ so that $w=Pw$
and that $P$ is symmetric:
\begin{align*}
I_B^3 &=
        \sum_{K_{h} \in \mathcal{T}_{h}}\int_{K^l_{h}} 
               \big(v_h^l  w \cdot (PP_h(I-dH)P-P)\nabla_{\Gamma} u \dA  \\
      &=
        \sum_{K_{h} \in \mathcal{T}_{h}}\int_{K^l_{h}} v_h^l
               w \cdot (PP_hP-P)\nabla_{\Gamma} u
             + d w \cdot P_hHP\nabla_{\Gamma} u \dA 
      \lesssim h^2\hbarnorm{v_h^l}\norm{u}_{H^2(\Gamma)}
\end{align*}
It remains to bound the term on the skeleton of the grid. Note first that
\[
  [w_h;n_h] = ((w_h^+-w^{-l})\cdot n_h^+) + ((w_h^--w^{-l})\cdot n_h^-)
              +w^{-l}\cdot(n_h^++n_h^-) 
            \lesssim h^{p_1}+h^2
\]
using our Assumptions on $w_h$ and scalling results from
Lemma~\eqref{Gamma2GammahSmall}. This leads to the following estimate
\begin{align*}
I_B^4 &\lesssim
      (h^{p_1}+h^2) \sum_{e_h\in\mathcal{E}_h}\norm{u^{-l}}_{L^2(e_h)}\norm{v_h}_{L^2(e_h)}
      \lesssim 
      (h^{p_1}+h^2)h^{-1}\norm{u^{-l}}_{L^2(\Gamma_h)}\norm{v_h}_{L^2(\Gamma_h)}
\end{align*}
This completes the proof.  
\end{proof}

\section{Construction of discrete velocity field}\label{sec:advectionConstruction}
We will now attempt to justify the assumptions we have made on $w_{h}$ by constructing a 
discrete velocity field which satisfies assumptions 
(\ref{eq:advectionGeometricEstimate1a}) 
and (\ref{eq:advectionGeometricEstimate2}). We will first discuss why we do not
simply take $w_h=w^{-l}$ and then describe two alternative approaches,

\subsection{Downward lift of velocity field}

Consider the simplest choice for a discrete velocity field given by
$w_h:=w^{-l}$. Note that due to the definition of our bilinear form we can
always write $P_hw_h$ instead of $w_h$, so using the projection onto the
tangent planes of the triangles is not required in the definition of the
discrete velocity field. 

In general one can neither expect this choice to lead to a divergence free
field on each triangle, nor that the normal jumps across edges will vanish,
i.e., $[w_h;n_h]=w^{-l}(n_h^++n_h^-)\neq 0$.
We added a number of terms to take into account that the velocity field
$w_h$ is not divergence free and has non continuous normal components over
the element edges.
Simply defining the discrete bilinear form by taking $w_h=w^{-l}$ in $\BDG$, 
dropping the extra $\Cw$ term and taking $\gamma_h=0$ does not lead to a
unconditionally stable scheme. In fact the matrix resulting from such a scheme will
not be positive-definite independently of $h$. To see this consider the
bilinear form
\begin{align*}
  \sum_{K_{h} \in \mathcal{T}_{h}}\int_{K_{h}}  - w^{-l}u_{h} \cdot \nabla_{\Gamma_{h}} v_{h} + u_{h} v_{h} \dAh +  
    \sum_{e_{h} \in \mathcal{E}_{h}}\int_{e_{h}} \widehat{w^{-l} u_{h}} [v_{h}] \dsh 
\end{align*}      
Integrating by parts and using Lemma \ref{eq:intByParts} as
in the proof of Lemma \ref{Lemma:boundStabGammakIPAdvection} choosing 
$u_{h} \equiv 0$ for every $K_{h} \in \mathcal{T}_{h}$ except for two elements
$K_{h}^{+}$ and $K_{h}^{-}$ for which $e_{h} = K_{h}^{+} \bigcap
K_{h}^{-}$. Furthermore, we can assume without loss of generality that $n_{h}^{+} = (-1,0,0)$, $n_{h}^{-} = (\cos(q), \sin(q), 0)$ with $q \in (0, 2 \pi)$. Note that unless $q = 0, 2 \pi$, we have that $n_{h}^{+} \not = - n_{h}^{-}$. The velocity $w^{-l}$ at $e_{h}$ is assumed to be $(-1,0,0)$, so that $w^{-l} \cdot n_{h}^{+} = 1 > 0$ and $w^{-l} \cdot n_{h}^{-} = -\cos(q) < 0$. Finally, we assume that $u_{h}^{+} = u_{h}^{-} = 1$ so that $[u_{h}] = 0$ on $e_{h}$. With these conditions, the stability of (\ref{eq:motivationEquation}) boils down to showing that
\[- \half \sum_{K_{h} \in \mathcal{T}_{h}} \int_{\partial K_{h}}\left( w^{-l} \cdot n_{K_{h}} \right) u_{h}^{2} \dsh \geq 0 \]
which, from Lemma \ref{eq:intByParts} and the above conditions, is equivalent to showing that
\[- \half \int_{e_{h}} [w^{-l} u_{h}^{2}; n_{h}] \dsh\geq 0. \] 
Notice that the numerical flux does not appear given that it is scaled with $[u_{h}] = 0$, and thus cannot influence the sign of the above quantity. Expanding the expression, we have that
\begin{align*}
- \half \int_{e_{h}} [w^{-l} u_{h}^{2}; n_{h}] \dsh &= - \half \int_{e_{h}} w^{-l} u_{h}^{2+} \cdot n_{h}^{+} + w^{-l} u_{h}^{2-} \cdot n_{h}^{-}\dsh = \half |e_{h}| \left( \cos(q) - 1 \right) < 0.
\end{align*}
Hence, in general, whenever $n_{h}^{+} \not = -n_{h}^{-}$, $h$-independent positive-definiteness of the matrix resulting from the scheme may not hold, regardless of the choice of the modified upwind flux.

So in fact both $\gamma_h$ and $\Cw$ are important terms to make the scheme
positive definite independent of $h$. Since the $\Cw$ will not vanish, our
error estimate indicate a suboptimal convergence rate (of course $p_1$ is arbitrarily
large), although as already pointed out, this is not confirmed by our numerical experiments. 
A more severe problem is, that evaluating $\gamma_h$ requires the computation of
$\nabla_{\Gamma_h}\cdot w^{-l}$, which requires derivatives of the lifting
operator and thus more information about the surface $\Gamma$ then we wish
to have in our numerical scheme. We therefore will not use this choice in
our numerical experiments, but tests indicate that this choice is
comparable to the other choices described in the following.

\subsection{Lagrange interpolation}

We can use a Lagrange interpolation of $w^{-l}$ on $\Gamma_h$ to define
$w_h$. This is easy to implement and evaluating both $\gamma_h$ and $\Cw$
is not problematic. According to our error analysis the approximation order
of linear finite elements is sufficient ($p_1=2$) but of course $\Cw$ will in 
general not vanish and thus our error estimate will be suboptimal. In our numerical
experiments we still observed an optimal rate while a piecewise constant interpolation 
does not lead to optimal results. We omitted the detail of these
experiments in this paper.

\subsection{Surface Raviart-Thomas interpolant}
Our next choice avoids the problem of suboptimality by constructing a
velocity field with $[w_h;n_h]=0$. To this end we make use of 
a \emph{Raviart-Thomas}-type interpolant of $w^{-l}$, 
which we will refer to as the surface Raviart-Thomas interpolant.

Let $F_{K_{h}}$ denote the mapping from the reference element $K$ to $K_{h}$. Then we have that $\nabla F_{K_{h}} = (e_{0}, e_{1}) \in \mathbb{R}^{3 \times 2}$ where $e_{0}$ and $e_{1}$ are two edges of $K_{h}$ intersecting at the vertex $x_{0}$. We first define the local spaces 
\[ \mathbb{P}_{RT}^{q}(K_{h}) := \left\{ s_{h}(x) := \nabla F_{K_{h}}(F_{K_{h}}^{-1}(x)) p\left( F_{K_{h}}^{-1}(x) \right) ,\ p \in [\mathbb{P}^{q}(K)]^2 \right\}.\]
We next define the local Raviart-Thomas space of order $q$ on $K_{h}$ to be given by
\[   RT^{q}(K_{h}) := \left\{ \bar{w}_{h}(x) := s_{h}(x) + (x-x_0) t\left(F_{K_{h}}^{-1}(x)\right) ,\ s_{h} \in \mathbb{P}_{RT}^{q}(K_{h}),\ t \in \mathbb{P}^{q}( K ) \right\}. \]
It is clear from the definition of $RT^{q}(K_{h})$ that any function $ \bar{w}_{h} \in RT^{q}(K_{h})$ for every $K_{h} \in \mathcal{T}_{h}$ is tangential to $\Gamma_{h}$. Using the convention that the conormal to $e_{h} \subset \partial K_{h}$ is $n_{h}^{+}$, 
the local degrees of freedom of $\bar{w}_{h} \in RT^{q}(K_{h})$ are given by
\begin{align}\label{eq:RTDOFa}
&\int_{e_{h}} \bar{w}_{h} \cdot n_{h}^{+} p_{q}\dsh\ \forall p_{q} \in \mathbb{P}^{q}(e_{h}), e_{h} \subset \partial K_{h}, \\
&\int_{K_{h}} \bar{w}_{h} \cdot \mathbf{p}_{q-1}\dsh\ \forall \mathbf{p}_{q-1} \in \mathbb{P}_{RT}^{q-1}(K_{h})  \label{eq:RTDOFb}.
\end{align}
We then define, for $w^{-l} \in [W^{1,\infty}(K_{h})]^{3}$, the local surface Raviart-Thomas interpolant of order $q$ to be $\Pi_{K_{h}}^{q} w^{-l} \in RT^{q}(K_{h})$ satisfying
\begin{align}\label{eq:RTDefa}
&\int_{e_{h}} \Pi_{K_{h}}^{q} w^{-l} \cdot n_{h}^{+} p_{q}\dsh = \int_{e_{h}} w^{-l} \cdot n_{e_{h}}^{+} p_{q}\dsh\ \forall p_{q} \in \mathbb{P}^{q}(e_{h}), e_{h} \subset \partial K_{h}, \\
&\int_{K_{h}} \Pi_{K_{h}}^{q} w^{-l} \cdot \mathbf{p}_{q-1}\dsh = \int_{K_{h}} w^{-l} \cdot \mathbf{p}_{q-1}\dsh\ \forall \mathbf{p}_{q-1} \in \mathbb{P}_{RT}^{q-1}(K_{h}) \label{eq:RTDefb}.
\end{align}
Here, the ``average'' conormals $n_{e_{h}}^{+/-}$ are given by $n_{e_{h}}^{+/-} := \pm \frac{\frac{1}{2}(n_{h}^{+}- n_{h}^{-})}{|\frac{1}{2}( n_{h}^{+}- n_{h}^{-})|}$. 

\begin{remark}
Notice that this definition differs from that of the local classical
Raviart-Thomas interpolant in the way we have defined the right-hand side
of (\ref{eq:RTDefa}). We have to use what we call the ``average'' conormals
$n_{e_{h}}^{+/-}$ instead of the standard conormals $n_{h}^{+/-}$ because
they satisfy $n_{e_{h}}^{+} = -n_{e_{h}}^{-}$.
From here on, we will refer to the local classical Raviart-Thomas interpolant by $\WT{\Pi}_{K_{h}}^{q} w^{-l}$.  
\end{remark}

\begin{lemma}\label{lemma:assumptionConfirmation1}
Let $\Pi_{K_{h}}^{q} w^{-l}$ be the local surface Raviart-Thomas interpolant of $w^{-l} \in [W^{1,\infty}(K_{h})]^{3}$, defined as in (\ref{eq:RTDefa})--(\ref{eq:RTDefb}). Then we have that 
\[\Pi_{K_{h}}^{q} w^{-l} \cdot n_{h}^{+} = -\Pi_{K_{h}}^{q} w^{-l} \cdot n_{h}^{-} \]
for each $e_{h} \in \mathcal{E}_{h}$.  
\end{lemma}
\begin{proof}
By (\ref{eq:RTDefa})--(\ref{eq:RTDefb}) and using the fact that $n_{e_{h}}^{+} = -n_{e_{h}}^{-}$, we have that

\begin{align*}
\int_{e_{h}} \Pi_{K_{h}}^{q} w^{-l} \cdot n_{h}^{+} p_{q}\dsh = \int_{e_{h}} w^{-l} \cdot n_{e_{h}}^{+} p_{q}\dsh &= - \int_{e_{h}} w^{-l} \cdot n_{e_{h}}^{-} p_{q}\dsh = -\int_{e_{h}} \Pi_{K_{h}}^{q} w^{-l} \cdot n_{h}^{-} p_{q}\dsh.  
\end{align*}
It follows that
\begin{align*}
\int_{e_{h}} \left( \Pi_{K_{h}}^{q} w^{-l} \cdot n_{h}^{+} +  \Pi_{K_{h}}^{q} w^{-l} \cdot n_{h}^{+}\right)  p_{q}\dsh = 0 
\end{align*}
for every $p_{q} \in \mathbb{P}^{q}(e_{h})$. By Proposition 3.2 in \cite{fortin1991mixed}, we have that $\Pi_{K_{h}}^{q} w^{-l} \cdot n_{h}^{+} \in \mathbb{P}^{q}(e_{h})$, which gives us the pointwise equality
\[\Pi_{K_{h}}^{q} w^{-l} \cdot n_{h}^{+} = -\Pi_{K_{h}}^{q} w^{-l} \cdot n_{h}^{-} \]
as required.
\end{proof} 


\begin{lemma}\label{lemma:assumptionConfirmation2b}
Let $\Pi_{K_{h}}^{q} w^{-l}$ be the local surface Raviart-Thomas interpolant of $w^{-l} \in [W^{1,\infty}(\Gamma_{h})]^{3}$ defined as in (\ref{eq:RTDefa})--(\ref{eq:RTDefb}) and let $\WT{\Pi}_{K_{h}}^{q} w$ be its local classical Raviart-Thomas interpolant. We then have that
\begin{align*}
\norm{\Pi_{K_{h}}^{q} w^{-l} - \WT{\Pi}_{K_{h}}^{q} w^{-l}}_{L^{\infty}(K_{h} \cup \partial K_{h})} \lesssim h^{2} 
\end{align*}
for each $K_{h} \in \mathcal{T}_{h}$.
\end{lemma}
\begin{proof}
Denote by $\{N_{i}^{\partial K_{h}}\}_{i = 1}^{n_{\partial K_{h}}}$ the set of local degrees of freedom given by (\ref{eq:RTDOFa}) and $\{\varphi_{i}^{\partial K_{h}}\}_{i = 1}^{n_{\partial K_{h}}}$ the associated (vector-valued) basis functions. Similarly, we denote by  $\{N_{i}^{K_{h}}\}_{i = 1}^{n_{K_{h}}}$ the set of local degrees of freedom given by (\ref{eq:RTDOFb}) and $\{\varphi_{i}^{K_{h}}\}_{i = 1}^{n_{K_{h}}}$ the associated (vector-valued) basis functions. The local degrees of freedom for the local standard Raviart-Thomas interpolant $\{\WT{N}_{i}\}_{i = 1}^{n_{\partial K_{h}}}$ and $\{\WT{N}_{i}\}_{i = 1}^{n_{K_{h}}}$ are defined similarly.

We then have that \[\Pi_{K_{h}}^{q} w^{-l}(x) = \sum_{i = 1}^{n_{\partial K_{h}}} N_{i}^{\partial K_{h}}(w^{-l})\varphi_{i}^{\partial K_{h}}(x) + \sum_{i = 1}^{n_{K_{h}}} N_{i}^{K_{h}}(w^{-l})\varphi_{i}^{K_{h}}(x),\] and similarly for $\WT{\Pi}_{K_{h}}^{q} w^{-l}$. Then by noting that $N_{i}^{K_{h}}(w^{-l}) = \WT{N}_{i}^{K_{h}}(w^{-l})$ and making use of (\ref{eq:RTDOFa}) and (\ref{eq:RTDefa}), we have that
\begin{align*}
&\norm{\Pi_{K_{h}}^{q} w^{-l} - \WT{\Pi}_{K_{h}}^{q} w^{-l}}_{L^{\infty}(K_{h} \cup \partial K_{h})} \\ 
&= \left\norm{\sum_{i = 1}^{n_{\partial K_{h}}}\left( N_{i}^{\partial K_{h}}(w^{-l}) - \WT{N}_{i}^{\partial K_{h}}(w^{-l})\right)\varphi_{i}^{\partial K_{h}}\right}_{L^{\infty}(K_{h} \cup \partial K_{h})} \\
&\leq \max_{1 \leq i \leq n_{\partial K_{h}}} \left| N_{i}^{\partial K_{h}}(w^{-l}) - \WT{N}_{i}^{\partial K_{h}}(w^{-l}) \right| \sum_{i = 1}^{n_{\partial K_{h}}} \left|\varphi_{i}^{\partial K_{h}}\right| \lesssim \max_{1 \leq i \leq n_{\partial K_{h}}} \left| N_{i}^{\partial K_{h}}(w^{-l}) - \WT{N}_{i}^{\partial K_{h}}(w^{-l}) \right| \\
&= \max_{1 \leq i \leq n_{\partial K_{h}}} \left| \int_{e_{h}} w^{-l} \cdot n_{e_{h}}^{+} \xi_{i}\dsh - \int_{e_{h}} w^{-l} \cdot n_{h}^{+} \xi_{i}\dsh  \right| = \max_{1 \leq i \leq n_{\partial K_{h}}} \left| \int_{e_{h}} w^{-l} \cdot \left( \mathbf{P}^{-l}n_{e_{h}}^{+} - \mathbf{P}^{-l}n_{h}^{+}\right) \xi_{i} \dsh \right| \\
&\lesssim \norm{n - \mathbf{P}n_{e_{h}}^{+}}_{L^{\infty}(\mathcal{E}_{h})} + \norm{n - \mathbf{P}n_{h}^{+}}_{L^{\infty}(\mathcal{E}_{h})} \lesssim h^{2}  
\end{align*}
where $\{ \xi_{i}\}$ denote the basis functions of $\mathbb{P}^{q}(e_{h})$. The last estimate follows from Lemma \ref{Gamma2GammahSmall}.
\end{proof}

The following theorem will help justify assumption (\ref{eq:advectionGeometricEstimate1a})
for the case of the local surface Raviart-Thomas interpolant of zero order ($q=0$).
\begin{theorem}\label{lemma:assumptionConfirmation2a3a}
Let $w^{-l} \in [W^{1,\infty}(K_{h})]^{3}$ and $\WT{\Pi}_{K_{h}}^{0} w^{-l}$ be its local classical Raviart-Thomas interpolant of zero order defined only through condition (\ref{eq:RTDefa}) (with $n_{e_{h}}^{+}$ replaced by $n_{h}^{+}$). We then have that
\begin{align*}
\norm{\mathbf{P}_{h} w^{-l} - \WT{\Pi}_{K_{h}}^{0} w^{-l}}_{L^{\infty}(K_{h})} &\lesssim h \norm{\nabla_{\Gamma_{h}}w^{-l}}_{L^{\infty}(K_{h})}, \\
\left\norm{\nabla_{\Gamma_{h}} \cdot \left( \mathbf{P}_{h} w^{-l} - \WT{\Pi}_{K_{h}}^{0} w^{-l}\right)\right}_{L^{\infty}(K_{h})} &\lesssim h | \nabla_{\Gamma_{h}} w^{-l} |_{W^{1,\infty}(K_{h})} 
\end{align*}
for each $K_{h} \in \mathcal{T}_{h}$.
\end{theorem}
\begin{proof}
The proof of the first estimate follows similar lines as that of Theorem 6.3 in \cite{acosta2011error}. The second estimate follows similar lines as that of Theorem 1.114 in \cite{ern2004theory}. 
\end{proof}

The first estimate of Theorem \ref{lemma:assumptionConfirmation2a3a} together with Lemma \ref{lemma:assumptionConfirmation2b}  guarantees that the local surface Raviart-Thomas interpolant also satisfies Theorem \ref{lemma:assumptionConfirmation2a3a}. As such, assumption (\ref{eq:advectionGeometricEstimate1a}) holds when choosing $w_{h}$ to be the local surface Raviart-Thomas interpolant of zero order. 


We finally show that assumption (\ref{eq:advectionGeometricEstimate2}) holds for local surface Raviart-Thomas interpolants of zero order.

\begin{lemma}\label{lemma:assumptionConfirmation3b}
Let $w^{-l} \in [W^{1,\infty}(K_{h})]^{3}$, $K_{h} \in \mathcal{T}_{h}$ , and $\Pi_{K_{h}}^{0} w^{-l}$ be its local surface Raviart-Thomas interpolant of zero order defined only through condition (\ref{eq:RTDefa}). We then have that
\[\left\norm{\nabla_{\Gamma_{h}} \cdot \Pi_{K_{h}}^{0} w^{-l}\right}_{L^{\infty}(K_{h})} \lesssim h . \]   
\end{lemma}
\begin{proof}
We have that
\begin{align*}
\left\norm{\nabla_{\Gamma_{h}} \cdot \Pi_{K_{h}}^{0} w^{-l}\right}_{L^{\infty}(K_{h})} \leq &\left\norm{\nabla_{\Gamma_{h}} \cdot \left( \Pi_{K_{h}}^{0} w^{-l} -  \mathbf{P}_{h} w^{-l}\right)\right}_{L^{\infty}(K_{h})}\\ 
&+ \norm{\nabla_{\Gamma_{h}} \cdot \mathbf{P}_{h} w^{-l}}_{L^{\infty}(K_{h})}.  
\end{align*}
Making use of Lemma 3.2 in \cite{olshanskiivolume}, we have that the second term scales like $h$. For the first term, we have that
\begin{align*}
\left\norm{\nabla_{\Gamma_{h}} \cdot \left( \Pi_{K_{h}}^{0} w^{-l} -  \mathbf{P}_{h} w^{-l}\right)\right}_{L^{\infty}(K_{h})} &\leq \left\norm{\nabla_{\Gamma_{h}} \cdot \left( \Pi_{K_{h}}^{0} w^{-l} -  \WT{\Pi}_{K_{h}}^{0} w^{-l}\right)\right}_{L^{\infty}(K_{h})}\\ 
&+ \left\norm{\nabla_{\Gamma_{h}} \cdot \left( \WT{\Pi}_{K_{h}}^{0} w^{-l} -  \mathbf{P}_{h} w^{-l}\right)\right}_{L^{\infty}(K_{h})}.
\end{align*}
The second term in the above scales appropriately by the second estimate of Theorem \ref{lemma:assumptionConfirmation2a3a}. For the first term we proceed as in the proof of Lemma \ref{lemma:assumptionConfirmation2b} to get that
\begin{align*}
&\left\norm{\nabla_{\Gamma_{h}} \cdot \left( \Pi_{K_{h}}^{0} w^{-l} -  \WT{\Pi}_{K_{h}}^{0} w^{-l}\right)\right}_{L^{\infty}(K_{h})} \\
&\leq \max_{1 \leq i \leq n_{\partial K_{h}}} \left| N_{i}^{\partial K_{h}}(w^{-l}) - \WT{N}_{i}^{\partial K_{h}}(w^{-l}) \right| \sum_{1 \leq i \leq n_{\partial K_{h}}} \left|\nabla_{\Gamma_{h}} \cdot \varphi_{i}^{\partial K_{h}}\right| \\ 
&\lesssim \max_{1 \leq i \leq n_{\partial K_{h}}}\left| N_{i}^{\partial K_{h}}(w^{-l}) - \WT{N}_{i}^{\partial K_{h}}(w^{-l}) \right| h^{-1} \lesssim h 
\end{align*}
as required.
\end{proof}

\section{Numerical tests}
For the test problems discussed below, we will focus on a surface IP discretisation of the diffusion term and call the resulting approximation the surface IP/UP approximation. Furthermore, the discrete velocity field $w_{h}$ is chosen to be the zero order surface Raviart-Thomas interpolant of $w^{-l}$ i.e. $w_{h}|_{K_{h}} = \Pi_{K_{h}}^{0}w^{-l}$. We will also briefly discuss the case when we choose $w_{h} = w^{-l}$ in the numerics.

\subsection{Test problem on torus}
Our first test problem, considered in \cite{olshanskiivolume}, involves solving 
(\ref{eq:weakH1AdvectionDiffusion}) on the torus
\[\Gamma = \left\{(x_{1}, x_{2}, x_{3}) \ | \ \left( \sqrt{x_{1}^{2} + x_{2}^{2}} -1 \right)^{2} + x_{3}^{2} = \frac{1}{16}\right\}\]
with velocity field
\[ w(x) =\frac{1}{\sqrt{x_{1}^{2} + x_{2}^{2}}}\left(-x_{2},x_{1}, 0\right)^{T}.  \]
Note that the velocity field $w$ is tangential to the torus and divergence-free. We set $\epsilon = 10^{-6}$ and construct the right-hand side $f$ such that the solution $u$ of 
(\ref{eq:weakH1AdvectionDiffusion}) is given by 
\[u(x) = \frac{x_{1} x_{2}}{\pi} \arctan \left(\frac{x_{3}}{\sqrt{\epsilon}}\right).  \]
Note that $u$ has a sharp internal layer as shown in Figure \ref{fig:advectionProblem}.

\begin{figure}[htp]
\begin{center}
\hspace{-2.0cm} \includegraphics[width=1.7in]{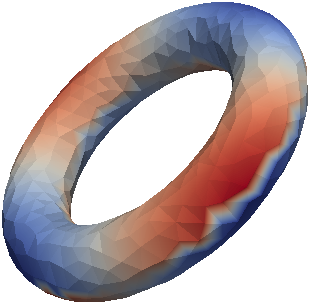}\\
\includegraphics[width=2.7in]{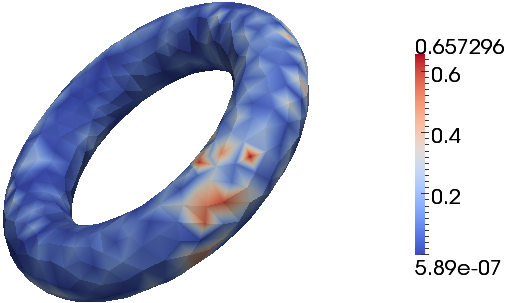}
\includegraphics[width=2.7in]{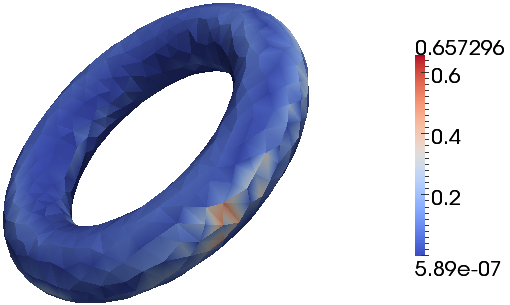}
\end{center}
\caption{Exact solution of (\ref{eq:weakH1AdvectionDiffusion}) (top) and pointwise errors for respectively the (unstabilised) surface FEM approximation (bottom left) and the surface IP/UP approximation (bottom right) on the torus (1410 elements).}
\label{fig:advectionProblem}
\end{figure}

Figure \ref{fig:advectionProblem} shows the exact solution and both the unstabilised surface FEM approximation and the surface IP/UP approximation of 
(\ref{eq:weakH1AdvectionDiffusion}). 
Notice how, as in the planar case, the unstabilised surface FEM approximation exhibits global spurious oscillations whilst the surface IP/UP approximation is completely free of such oscillations. We obtain similar results for the case when we choose $w_{h} = w^{-l}$ in the surface IP/UP method, although $L^{\infty}$ errors tend to be slightly larger for such a choice.  

\subsection{Test problem on sphere}
Next, we consider (\ref{eq:weakH1AdvectionDiffusion}) on the unit sphere
\[\Gamma = \{ x \in \mathbb{R}^{3}\ : \ |x| = 1 \} \]
with velocity field
\[ w(x) =\left(-x_{2}\sqrt{1-x_{3}^{2}},x_{1}\sqrt{1-x_{3}^{2}}, 0\right)^{T}.  \]
Notice again that $w$ is tangential to the sphere and divergence-free. We set $\epsilon = 10^{-6}$ and construct the right-hand side $f$ such that the solution $u$ of 
(\ref{eq:weakH1AdvectionDiffusion}) is given by the expression given in the previous test problem. Tables \ref{tab:advectionProblem1FEM} and \ref{tab:advectionProblem1DG} show the $L^{2}$ and DG norm errors/EOCs outside the sharp internal layer, given by $D = \{ x \in \Gamma \ : \ |x_{3}| > 0.3\}$, for respectively the (unstabilised) surface FEM approximation and the surface IP/UP approximation.

\begin{table}[htp]
\begin{center}
\begin{tabular}{|c|c|cc|cc|}
\hline
Elements&$h$
    &$L_{2}(D)$-error&$L_{2}(D)$-eoc
    &$DG(D)$-error&$DG(D)$-eoc\\
\hline
632   &0.2239&0.04462&    &0.865 &\\
2528  &0.1121&0.01736&1.36&0.652 &0.40\\
10112 &0.0561&0.00936&0.89&0.727 &-0.16\\
40448 &0.0280&0.00604&0.63&0.934 &-0.36\\
161792&0.0140&0.00356&0.76&1.095 &-0.23\\
647168&0.0070&0.00169&1.07&1.038 &0.08\\
\hline
\end{tabular}
\end{center}
\caption{Errors and convergence orders for the (unstabilised) surface FEM
approximation of (\ref{eq:weakH1AdvectionDiffusion}) on the subdomain $D$ of the unit sphere.}
\label{tab:advectionProblem1FEM}
\end{table}

\begin{table}[htp]
\begin{center}
\begin{tabular}{|c|c|cc|cc|}
\hline
Elements&$h$
    &$L_{2}(D)$-error&$L_{2}(D)$-eoc
    &$DG(D)$-error&$DG(D)$-eoc\\
\hline
632   &0.2239&0.0073256&    &0.15932 &\\
2528  &0.1121&0.0021745&1.75&0.08892 &0.84\\
10112 &0.0561&0.0006499&1.75&0.05015 &0.83\\
40448 &0.0280&0.0001917&1.76&0.02820 &0.83\\
161792&0.0140&5.399e-05&1.83&0.01537 &0.88\\
647168&0.0070&1.394e-05&1.95&0.00778 &0.98\\
\hline
\end{tabular}
\end{center}
\caption{Errors and convergence orders for the IP/UP approximation of 
(\ref{eq:weakH1AdvectionDiffusion}) on the subdomain $D$ of the unit sphere.}
\label{tab:advectionProblem1DG}
\end{table}

The results clearly indicate that the surface IP/UP method performs better than the unstabilised surface FEM. The results for the surface IP/UP method indicate a $O(h^2)$ convergence in the $L^{2}(D)$-norm and $O(h)$ in the $DG(D)$-norm. The unstabilised surface FEM, on the other hand, shows a much more erratic behaviour and does not attain its asymptotic convergence rates within our computational domain.   

\begin{table}[htp]
\begin{center}
\begin{tabular}{|c|c|cc|cc|}
\hline
Elements&$h$
    &$L_{2}(D)$-error&$L_{2}(D)$-eoc
    &$DG(D)$-error&$DG(D)$-eoc\\
\hline
632   &0.2239&0.0040846&    &0.11275&\\
2528  &0.1121&0.0010464&1.96&0.05707&0.98\\
10112 &0.0561&0.0002654&1.98&0.02867&0.99\\
40448 &0.0280&6.679e-05&1.99&0.01437&1.00\\
161792&0.0140&1.670e-05&2.00&0.00718&1.00\\
647168&0.0070&4.161e-06&2.00&0.00359&1.00\\
\hline
\end{tabular}
\end{center}
\caption{Errors and convergence orders for the IP/UP approximation of
(\ref{eq:weakH1AdvectionDiffusion})  with $w_{h} = w^{-l}$ on the subdomain $D$ of the unit sphere.}
\label{tab:advectionProblem1DGOld}
\end{table}

Table \ref{tab:advectionProblem1DGOld} show the relevant errors when using $w_{h} = w^{-l}$ in the surface IP/UP approximation. The errors appear to be smaller by a factor of about $0.5$ compared to those shown in Table \ref{tab:advectionProblem1DG} for which we chose $w_{h}|_{K_{h}} = \Pi_{K_{h}}^{0}w^{-l}$. This can be explained by the fact that triangulations for simple surfaces such as the unit sphere can be constructed to be very ``smooth'' (in the sense that the relation $n_{h}^{+} = -n_{h}^{-}$ practically holds for each $e_{h} \in \mathcal{E}_{h}$) and that the zero order Raviart-Thomas approximation error is relatively large.     

\section{Conclusions}

It is well known that the DG method is especially well suited for
stabilizing transport terms in PDE models. This type of problem has not yet
been studied using discrete surface finite-elements. In this paper we proved a-priori error
estimate for the DG method for the stationary linear hyperbolic problems and elliptic problems with possibly 
dominate advection term. This extends previous results for the laplace equations on surfaces.
The theory and the numerical experiments show that special care has to be
taken when projecting the given continuous velocity field unto the discrete
surface. The main two problems identified are the non zero divergence and
the jump in the normal component. Different suggestions were discussed on
how to handle these problems either by extending the bilinear form to take
the jump in the velocity field into account or by the use of a special
projection based on Raviart-Thomas like interpolation operators. Numerical
experiments demonstrate the accuracy of the resulting method.

\section*{Acknowledgements}
This research has been supported by the British Engineering and Physical Sciences Research Council (EPSRC), Grant EP/H023364/1.

\bibliographystyle{IMANUM-BIB}
\bibliography{IMANUM-refs}

\end{document}